\documentclass[journal,twoside,web]{ieeecolor}

 \usepackage{amsthm}
\usepackage{amsopn}
\usepackage{generic}
\usepackage{cite}
\usepackage{amsmath}
\usepackage{amssymb}
\usepackage{tikz}
\usepackage{amsfonts}
\usepackage{graphicx}
\usepackage{algorithm,algorithmic}
\usepackage{hyperref}
\usepackage{textcomp}
\usepackage{comment}

\usepackage{url}
\usepackage{diagbox}
\hypersetup{
    colorlinks=true,
    linkcolor=blue,
    filecolor=blue,      
    urlcolor=blue,
    citecolor=cyan,
}

\newtheorem{theorem}{Theorem}[]

\newtheorem{lemma}{Lemma}[]
\newtheorem{remark}{Remark}

\newtheorem{definition}{Definition}[]

\newtheorem{corollary}{Corollary}[]

\newcommand{\K}{\mathbf{K}}

\newcommand{\U}{\mathbf{U}}
\newcommand{\X}{\mathbf{X}}
\newcommand{\W}{\mathbf{W}}

\newcommand{\z}{\mathbf{z}}

\newcommand{\Q}{\mathbf{Q}}
\newcommand{\R}{\mathbf{R}}
\newcommand{\Rb}{\mathbb{R}}
\newcommand{\bu}{\mathbf{u}}
\newcommand{\bv}{\mathbf{v}}
\newcommand{\bw}{\mathbf{w}}
\newcommand{\f}{\mathbf{f}}
\newcommand{\F}{\mathbf{F}}
\newcommand{\x}{\mathbf{x}}

\newcommand{\bz}{\mathbf{z}}
\newcommand{\A}{\mathbf{A}}
\newcommand{\B}{\mathbf{B}}
\newcommand{\C}{\mathbf{C}}
\newcommand{\D}{\mathbf{D}}
\newcommand{\E}{\mathbf{E}}

\newcommand{\Eb}{\mathbb{E}}

\newcommand{\bP}{\mathbf{P}}
\newcommand{\Pb}{\mathbb{P}}
\newcommand{\Hb}{\mathbf{H}}
\newcommand{\Hc}{\mathcal{H}}
\newcommand{\I}{\mathbf{I}}
\newcommand{\M}{\mathbf{M}}
\newcommand{\Nb}{\mathbb{N}}
\newcommand{\Nbp}{\mathbb{N}_+}

\newcommand{\Sb}{\mathbf{S}}

\newcommand{\V}{\mathbf{V}}
\newcommand{\Y}{\mathbf{Y}}
\newcommand{\Z}{\mathbf{Z}}

\DeclareMathOperator*{\lqr}{lqr}
\definecolor{blue-violet}{rgb}{0.54, 0.17, 0.89}

\def\BibTeX{{\rm B\kern-.05em{\sc i\kern-.025em b}\kern-.08em
    T\kern-.1667em\lower.7ex\hbox{E}\kern-.125emX}}
\begin{document}
\title{Noise Sensitivity of the Semidefinite Programs for Direct Data-Driven LQR}
\author{Xiong Zeng, Laurent Bako, and Necmiye Ozay \IEEEmembership{Senior Member, IEEE}
\thanks{This work is supported in part by ONR CLEVR-AI MURI (\#N00014-21-1-2431) and NSF CNS Grant \#1931982.}
\thanks{Xiong Zeng and Necmiye Ozay are with the Department of Electrical Engineering and Computer Science, University
of Michigan Ann Arbor, MI 48105 USA. (e-mails: \{zengxion, necmiye\}@umich.edu). }
\thanks{Laurent Bako is with Ecole Centrale de Lyon, INSA Lyon, Universite Claude
Bernard Lyon 1, CNRS, Ampere, UMR 5005, 69130 Ecully, France. (email:
laurent.bako@ec-lyon.fr).}
 }

\maketitle

\begin{abstract} 
  In this paper, we study the noise sensitivity of the semidefinite program (SDP) proposed for direct data-driven infinite-horizon linear quadratic regulator (LQR) problem for discrete-time linear time-invariant systems. While this SDP is shown to find the true LQR controller in the noise-free setting, we show that it leads to a trivial solution with zero gain matrices when data is corrupted by noise, even when the noise is arbitrarily small. We then study a variant of the SDP that includes a robustness promoting regularization term and prove that regularization does not fully eliminate the sensitivity issue. In particular, the solution of the regularized SDP converges in probability also to a trivial solution. 
\end{abstract}


\section{Introduction}

Certainty equivalence approach and robust control approach are two alternative paradigms in learning-based control. Roughly speaking, in certainty equivalence, we \emph{pretend} that our data is not corrupted by noise, the estimated model is the true system model, or the estimated control policy is designed based on the true system and clean data. Whereas, in robust control approach, we try to bound the effect of the noise in the data and aim to find a controller that achieves the desired properties for all possible noise values within this bound.

 These two different paradigms can be applied both in the model-based setting, where system identification is followed by control design, or in direct data-driven control, where data is used directly to synthesize a controller utilizing ideas from the behavioral system theory (see, e.g.,\cite{de2019formulas,martin2023guarantees}).  
In the context of model-based LQR, Mania et al. \cite{mania2019certainty} show that certainty equivalence is statistically consistent and is more sample-efficient than the robust approach given in \cite{dean2020sample}. The success of certainty equivalent control, in this case, lies in the fact that there is some inherent robustness in the solutions of the Riccati equations with respect to perturbations in system matrices. That is, small perturbations in system matrices result in small changes in the corresponding optimal controller (cf., Fig.~\ref{diagram}).

\begin{figure}
    \centering
{ \begin{tikzpicture}[font=\small]
\tikzstyle{block} = [draw, fill=white, rectangle, rounded corners, minimum height=2em, minimum width=3em]
\tikzstyle{label_style} = [font=\bfseries\color{blue!70}]

\node (data) at (-0.6,0) {  };
 \node  (M) at (0.9,0) {  };

\node[block, minimum width=4.5em, align=center] (sysid) at (2.1,1.1) {System Id};
\node[block, minimum width=7.5em, align=center] (LQR) at (5.1,1.1) { LQR$(\hat{\A}, \hat{\B}, \Q, \R)$};
\node[block, minimum width=7.5em, align=center] (sdp) at (4,-1.1) {DDD LQR via \eqref{CEDDDLQRSDP}};
\node (kce) at (7.3,-1.1) { $\K_{ce}$};
\node (kmb) at (7.3,1.1) { $\K_{mb}$};

\node at (4,2.1) {\textcolor{blue!70}{\textbf{Model-based}}};
\node at (4,-0.1) {\textcolor{red!70}{\textbf{Direct data-driven}}};

 \draw [-, thick] (data) -- node[above] {\hspace{-11pt}\small$\{ \bu_t, \x_t \}_{t=0}^T$} (0.8,0);
 \draw [->, thick] (0.8,0) |- (sysid.west);
\draw [->, thick] (sysid) -- node[above] {\tiny$(\hat{\A}, \hat{\B})$} (LQR);
\draw [->, thick] (LQR)  -- node[above] { } (kmb);
\draw [->, thick] (0.8,0) |- (sdp.west);
\draw [->, thick] (sdp) -- node[above] { } (kce);

\draw[thick, blue!70 ] (1.1 ,0.4) rectangle (6.6,1.8);
\draw[thick, red!70 ] (1.1 ,-1.8) rectangle (6.6,-0.4);

\end{tikzpicture}}
   \caption{Model-based (or indirect) and direct data-driven control algorithms give the same LQR gain, which is equal to the true LQR gain when the input data $\{ \bu_t, \x_t \}_{t=0}^T$ is persistently exciting and comes from a noise-free system. The model-based algorithm is continuous (indeed locally Lipschitz continuous) with respect to its inputs (a property also known as algorithmic robustness), therefore its output degrades gracefully with a change in the input \cite{mania2019certainty}. In this paper, we show that the direct data-driven LQR algorithm is discontinuous and even with arbitrary small noise (in almost all directions), the resulting gain reduces to the trivial gain of zero. }
  \label{diagram}
\end{figure}
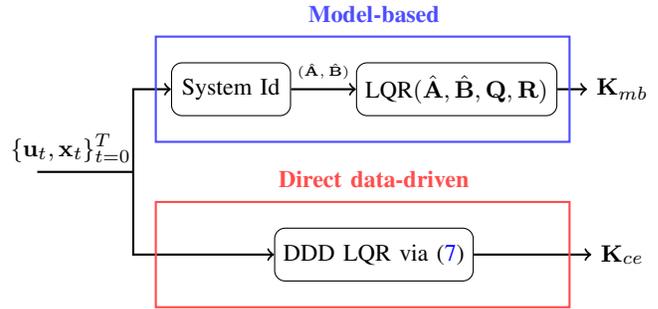

A natural question is how the certainty equivalence approach and a robust control approach compare in terms of statistical properties for direct data-driven LQR. Several works consider a robust approach for direct data-driven control for different control objectives or noise settings (e.g., \cite{van2020noisy, dai2023data,berberich2020robust}). On the other hand,  De Persis and Tesi \cite{de2021low} analyze a certainty equivalent approach to direct data-driven LQR, where they provide a sufficient condition for stabilizability and regularization techniques for improving noise robustness. Here, the certainty equivalence approach amounts to using the noisy data directly in the semidefinite programs developed for the noise-free case.  It is observed in \cite{de2021low} that even small noise can lead to a violation of their proposed sufficient condition for stabilizability and the semidefinite program may favor low gain solutions. However, a thorough understanding of the noise sensitivity and statistical properties of these semidefinite programs for direct data-driven LQR is missing. In this paper, we show that the semidefinite program for direct data-driven LQR is very sensitive to noise and yields, with probability one, \emph{trivial control gains} independent of the data even when there is an arbitrarily small amount of noise. Moreover, the robustified version also suffers from a similar issue as the length of the data trajectory used in the program goes to infinity. Therefore, neither of these approaches is statistically consistent. 

It is worth noting that there are recent works that propose alternative optimization formulations for direct data-driven control \cite{dorfler2023certainty,zhao2024data}, solutions of which mimic the solution of the model-based certainty equivalent control. Since the controllers synthesized by these alternative formulations are equivalent to model-based certainty equivalent control, they inherit the nice statistical consistency properties of the former. Our analysis does not pertain to these alternative formulations. Similarly, our results are not directly related to the finite-horizon data-driven control problems for which some connections with model-based approaches are established \cite{fiedler2021relationship,mattsson2024equivalence} since we focus on approximation of the infinite-horizon LQR gain.


{A preliminary version of this paper has been submitted to \cite{zeng2024dddlqr}. Compared with the conference version, we extend the results of certainty equivalent DDD LQR from scalar systems to multivariate systems and we also show that when the length of the data trajectory approaches infinity, the robustness promoting DDD LQR from \cite{de2021low} will also yield a zero state feedback gain estimate. }


The remainder of the paper is organized as follows. In Section \ref{sec_preliminaries}, we review some basic definitions from probability and data-driven control. Section \ref{sec_DDD_LQR} reviews two formulations of the direct data-driven LQR problem from the literature. Section \ref{sec_CE_DD_LQR} and
Section \ref{sec_RP_DDD_LQR} introduces the main results of the paper together with their proofs. We provide some numerical examples in Section~\ref{sec:num_ex} before concluding the paper in Section~\ref{sec_conclusion_future_work}.

\textbf{Notation.} We use lower case, lower case boldface, and upper case boldface letters to denote scalars,
vectors, and matrices respectively. For a matrix $\X$, $\underline{\sigma}(\X)$ and $\overline{\sigma}(\X)$  denote the least and largest singular value of $\X$, respectively. We denote by  $\det(\X)$ the determinant of $\X$, and $\X^\top$ denotes the transpose of $\X$. For a square matrix $\M$, $\rho(\M)$ denotes its spectral radius, $\mathbf{M}\succ 0$ ($\succeq 0$) denotes that $\mathbf{M}$ is symmetric and positive definite (positive semidefinite). $\A\succeq \B$ denotes that $\A-\B \succeq 0$. For a matrix $\M\succeq 0$, $\sqrt{\M}$ denotes the   matrix square root of $\M$, that is the unique matrix $\mathbf{F}\succeq 0$ such that $\mathbf{F}^2=\M$. $\I_n$ denotes the identity matrix whose dimension is $n\times n$. We use $\mathbb{N}$ to denote the set $\{0,1,2,\dots \}$, use $\Nbp$ to denote the set $\{1,2,\dots \}$, and use $[T]$ to denote the set $\{0,1,\dots,T \}$ for $T\in \Nb$.  $\mathcal{N} $ denotes the Gaussian distribution. $\mathbb{R}$ denotes the real number domain. 
For a vector-valued sequence $\f_0,\f_1, \ldots$, with $\f_t\in \Rb^s$ for $t\in \mathbb{N}$, we define the matrix $\F_i(j):=\left[\f_i \; \f_{i+1} \;  \ldots \; \f_{i+j-1} \right] \in \Rb^{s\times j} $.

\section{Preliminaries }
\label{sec_preliminaries}
\subsection{Probability fundamentals}
\begin{definition}[Convergence in Probability]
    A sequence $\left\{X(T)\right\}_{T=0}^{\infty}$ of random variables converges in probability towards the random variable $X$ if for all $\delta>0$

$$
\lim _{T \rightarrow \infty} \mathbb{P}\left(\left|X(T)-X\right|>\delta\right)=0,
$$
which is denoted as $X(T) \overset{p}{\rightarrow} X$ for short.
\end{definition}

The definition trivially generalizes to vector- and matrix-valued random variables by considering them element-wise.

\subsection{Results from Data-Driven Control}

The following notion will be relevant in establishing our results.

 \begin{definition}[Persistency of Excitation]
  \label{persistencyOfExcitation}
  Let $T \in \Nbp$. A sequence $\F_0({ T )} \in \Rb^{s\times T}$ is called persistently exciting of order $k \in \Nbp$ if the Hankel matrix\begin{equation}
\begin{aligned}
  &  \mathcal{H}_k\left(\F_0({ T )}\right):=\\
  &\left[\begin{array}{cccc}
\f_0 & \f_{1} & \cdots & \f_{T-k} \\
\f_{1} & \f_{2} & \cdots & \f_{T-k+1} \\
\vdots & \vdots & \ddots & \vdots \\
\f_{k-1} & \f_{k} & \cdots & \f_{T-1 }
\end{array}\right]\in \Rb^{ks\times (T-k+1)},
\label{Hankel}
\end{aligned}
\end{equation}
has full row rank, i.e., $rank(\mathcal{H}_k\left(\F_0({ T )}\right)) = ks$.
  \end{definition}

\section{Direct Data-Driven (DDD) LQR} 
\label{sec_DDD_LQR}
We consider the following   discrete-time linear time-invariant (LTI) system in this paper:
\begin{equation}
 \x_{t+1} = \A \x_{t} + \B \bu_t + \bw_t,
     \label{LTI}
 \end{equation}
 where $\x_t \in \mathbb{R}^n $, $\bu_t \in \mathbb{R}^m $, $\bw_t \in \mathbb{R}^n $ are the state, input, and noise at time $t$, respectively. We assume that the initial state and the noise are independent random variables with distributions satisfying $\bw_t \overset{i.i.d.}{\sim} \mathcal{N}(0, \sigma_{w}^2 \I_n )$, $\Eb[\x_0]=0$, and   $\Eb [\x_0^2] = \sigma^2_{x_0}\I_n$.
 The system parameters $\A$ and $\B$ are unknown. The standard LQR problem for the  LTI system \eqref{LTI} is
\begin{equation}
\begin{aligned}
& \min _{\bu_0, \bu_1, \cdots} \lim_{T \rightarrow \infty} \mathbb{E}\left[\frac{1}{T} \sum_{t=0}^T\left(\x_t^\top \Q \x_t  + \bu_t^\top \R \bu_t   \right)\right] \\
& \text { s.t. } \text{Dynamics in} \;
\eqref{LTI},
\end{aligned}
\label{LQR}
\end{equation}
where
$\Q\succ 0$, $\R \succ 0$, and $\mathbb{E}$ denotes the expectation over the randomness from the initial state $\x_0$ and the process noise $\bw_t$. We assume  $(\A, \B)$ is controllable throughout the paper.
When $\A$ and $\B$ are known, this LQR problem can be solved by finding the positive definite solution $\bP$ of the discrete-time algebraic Riccati equation:  
\begin{equation}
\bP = \A^{\top} \bP  \A - \A^{\top} \bP  \B \left(\R+ \B^{\top} \bP  \B \right)^{-1} \B^{\top} \bP  \A + \Q.
\label{DARE}
\end{equation}
Then, the solution of \eqref{LQR} is $\bu_t=-\K_{\lqr} \x_t$ with the optimal state feedback gain given by 
\begin{equation}
\K_{\lqr} = -\left(\R + \B^{\top} \bP \B \right)^{-1} \B^{\top} \bP  \A .
\label{KLQR}
\end{equation}

In direct data-driven control, the parameters $\A$ and $\B$ are unknown and the goal is to directly estimate $\K_{\lqr}$ from data without explicitly estimating $\A$ and $\B$. We assume the data is collected offline by driving the system with a random input such that $\bu_t \overset{i.i.d.}{\sim} \mathcal{N}(0, \sigma_{u}^2  \I_m )$ with {$\sigma_u>0$}, and this input is independent of the noise process and the initial condition. Let us define the data matrices formed from this trajectory data collected over some horizon $T$:  
\begin{equation}
\begin{aligned}
\X_0(T)   & =\left[\begin{array}{llll}
\x_0 & \x_1 & \ldots & \x_{T-1}
\end{array}\right] \in \mathbb{R}^{n \times T}, \\
\U_0(T )    & =\left[\begin{array}{llll}
\bu_0 & \bu_1 & \ldots & \bu_{T-1}
\end{array}\right] \in \mathbb{R}^{m \times T}, \\
\X_1(T)  & =\left[\begin{array}{llll}
\x_1 & \x_2 & \ldots & \x_{T}
\end{array}\right] \in \mathbb{R}^{n \times T}.
\end{aligned}
\label{data_matrix}
\end{equation}
We similarly define $\W_0(T)     =\left[\begin{array}{llll}
\bw_0 & \bw_1 & \ldots & \bw_{T-1}
\end{array}\right] \in \mathbb{R}^{n \times T}$ for the (unknown) noise sequence.
Most of the time, when  $T$ is clear from context, we will drop it and use $\X_0$, $\U_0$, $\X_1$, and $\W_0$  to indicate the above quantities. 

De Persis and Tesi \cite{de2019formulas} proposed a  semidefinite program, called the DDD LQR, to estimate the optimal LQR gain:

 \begin{equation}
 \label{CEDDDLQRSDP}
\begin{array}{cl}
\underset{\X,  \Y}{\operatorname{minimize}} & \operatorname{trace}\left(\Q \X_0 \Y  \right)+\operatorname{trace}\left(  \X  \right) \\
\text { subject to } & 
\left[\begin{array}{cc}
\X_{0} \Y - \I_n &  \X_{1}  \Y \\
\Y^{\top} \X_{1}^{\top} & \X_{0} \Y
\end{array}\right] \succeq 0\\
&\left[\begin{array}{cc}
\X &  \sqrt{\R} \U_0 \Y  \\
\left(\sqrt{\R} \U_0 \Y \right)^{\top} & \X_0 \Y
\end{array}\right] \succeq 0,
\end{array}
\end{equation}
where $\X\in\Rb^{m\times m}$ and $  \Y\in\Rb^{T\times n}$. If $\Y^*_{ce}$ denotes an optimal solution of \eqref{CEDDDLQRSDP}, then the LQR gain estimate is given by 
\begin{equation}
\label{eq:kce}
\K_{ce} := -\U_0 \Y^*_{ce} (\X_0 \Y^*_{ce})^{-1}.
\end{equation} 
The following result from \cite{de2019formulas} shows that when the system does not have any noise (i.e., $\bw_t=0$ for all $t$) and a persistency of excitation condition holds (which can be shown to hold with probability $1$ when the input is Gaussian as assumed), we have $ \K_{ce} = \K_{\lqr} $. 

\begin{theorem}[Theorem 4 in \cite{de2019formulas}]
\label{thm:NoiselessIsGood}
Consider the data matrices $\X_0$, $\U_0$, $\X_1$ defined above and the feedback gain \eqref{eq:kce}. 
If  $\operatorname{rank}\left(\left[\begin{array}{l}
\X_0  \\
\U_0
\end{array}\right]\right)=m+n$ and $\bw_t=0$ for all $t$, then $\K_{ce} =\K_{\lqr} $.
\end{theorem}

The estimate $\K_{ce}$ is also known as a certainty equivalent DDD LQR solution \cite{de2021low,de2019formulas}. While the quality of the estimate $\K_{ce}$ is not well-understood in the noisy setting in general, the following sufficient condition is derived in \cite{de2021low} to ensure that the gain $\K_{ce}$  is stabilizing.

\begin{lemma}[Lemma 4 in \cite{de2021low}]  
\label{lem:PreviousSensitiveLemma}
Consider an optimal solution $\Y^*_{ce}$ of \eqref{CEDDDLQRSDP}.
    Let 
    $$\M:=\Y_{ce}^* ( \X_0 \Y_{ce}^*)^{-1} \left(\Y_{ce}^*\right)^\top,$$ and 
    $$ \Psi :=  \W_0 \M  \W_0^\top -  \X_1 \M  \W_0^\top-  \W_0 \M  \X_1^\top.$$
   If there exists $\eta \geq 1$ such that
$$
 \Psi \preceq  \left( 1-\frac{1}{\eta} \right) \I_n, 
$$
then   the state feedback gain $\K_{ce}$ from \eqref{eq:kce} is stabilizing  for the system \eqref{LTI}.
\end{lemma}

Our first main result (Theorem~\ref{th:CE_inconsistency}) is to show that even when the noise is arbitrarily small, $\K_{ce}$ can be an arbitrarily bad estimate.
We will also provide an interpretation of the above sufficient condition for stabilization in light of our Theorem \ref{th:CE_inconsistency}.

Given the observation that $\K_{ce}$ might not always give reasonable results with noisy data, in \cite{de2021low}, a regularized version of the problem control synthesis problem is proposed to increase the robustness of the DDD approach to noise. In particular, the DDD LQR with robustness promoting regularization is formulated as the following semidefinite program \cite{de2021low}:
\begin{equation}
\begin{array}{cl}
\underset{\X,  \Y,\Sb  }{\operatorname{minimize}} & \operatorname{trace}\left(\Q \X_0 \Y\right) + \operatorname{trace}\left(\X\right) + \eta \operatorname{trace}\left(\Sb\right)  \\
\text { subject to } 
& \left[\begin{array}{cc}
\X_{0} \Y - \I_n & \X_{1}  \Y \\
\Y^{\top} \X_{1}^{\top} & \X_{0} \Y
\end{array}\right] \succeq 0\\
&\left[\begin{array}{cc}
\X &  \sqrt{\R}  \U_0 \Y \\
\left( \sqrt{\R} \U_0 \Y \right)^{\top} & \X_0 \Y
\end{array}\right] \succeq 0\\
&\left[\begin{array}{cc}
\Sb &     \Y \\
\Y^{\top} & \X_0 \Y
\end{array}\right] \succeq 0,
\end{array}
\label{RPDDDLQRSDP}
\end{equation}
where $\X\in\Rb^{m\times m},  \Y\in\Rb^{T\times n}$, $\Sb\in \Rb^{T\times T}$, and $\eta$ is a positive regularization constant. Similarly, the data-driven optimal feedback gain based on this optimization problem is 
\begin{equation}\label{eq:Krp}
\K_{rp}:= - \U_0  \Y^*_{rp} \left(\X_0  \Y^*_{rp}\right)^{-1},
\end{equation} where $\Y^*_{rp}$ is an optimal solution of \eqref{RPDDDLQRSDP}.

Our second main result (Theorem~\ref{th:RP_inconsistency}) is to show that as the number of data points (i.e., $T$) goes to infinity, $\K_{rp}$ can also be an arbitrarily bad estimate of the optimal LQR feedback gain $\K_{\lqr}$ when the system is subject to noise.

 \section{Certainty Equivalent (CE) DDD LQR}
\label{sec_CE_DD_LQR}
\subsection{Inconsistency of CE DDD LQR}
\label{sec:InconsistencyCEDDDLQR}
Our first main result shows that in the presence of noise, with probability $1$, the solution of the certainty equivalence DDD LQR problem in \eqref{CEDDDLQRSDP} is independent of the data.

\begin{theorem}
\label{th:CE_inconsistency} Consider the LTI system in \eqref{LTI}.
Assume  {$\sigma_w>0$} and $T\geq (m+n)(n+1)+n.$ Then, for all optimal solutions $\Y_{ce}^*$ of \eqref{CEDDDLQRSDP}, the state feedback gain estimate $\K_{ce}$ given in \eqref{eq:kce} is unique and we have that
$$\Pb_T(\K_{ce}=\mathbf{0}_{m\times n})=1,$$
where the probability $\Pb_T$ is with respect to the randomness of $x_0$, $\U_0(T)$, and $\W_0(T)$.
\end{theorem}

Based on Theorem \ref{th:CE_inconsistency}, we have that when $T\geq (m+n)(n+1)+n$, the data-driven feedback gain based on certainty equivalence DDD LQR in \eqref{CEDDDLQRSDP} is equal to $\mathbf{0}_{m\times n}$ with probability one. Hence, $\K_{ce}$ will not converge to $\K_{\lqr}$ no matter how large $T$ is, i.e., $\K_{ce}$ is an inconsistent estimator of $\K_{\lqr}$. The proof of this theorem is given in Section \ref{ProofOfTheorem1}.

 \begin{remark} Results similar to
 Theorem \ref{th:CE_inconsistency} can be established for a linear system subject to measurement noise, that is:
 \begin{equation}
     \begin{aligned}
         \x_{t+1} &=\A\x_t +\B \bu_t\\
         \x^m_{t} &= \x_t + \mathbf{\delta}_t,
     \end{aligned}
 \end{equation}
 where $\mathbf{\delta}_t \overset{i.i.d.}{\sim} \mathcal{N}(0, \sigma_{\delta}^2\I_n)$ is the measurement noise. Now, if we form the data matrices in \eqref{data_matrix} using the measurements $\x^m_t$ in place of $\x_t$ and solve \eqref{CEDDDLQRSDP}, the resulting gain will again be zero with probability one. 
    \end{remark}

 Next, we present a corollary that provides an alternative interpretation of Lemma~\ref{lem:PreviousSensitiveLemma}.

\begin{corollary} \label{coro:PreviousGuaranteeViolation} Consider the matrix $\Psi$ defined in Lemma \ref{lem:PreviousSensitiveLemma}. Under the premises of Theorem \ref{th:CE_inconsistency}, we have that  $\Pb_T \left(\Psi = \A\A^\top\right)=1$ with $\A$ being the system matrix from \eqref{LTI}.
\end{corollary}
    Corollary \ref{coro:PreviousGuaranteeViolation} shows that when the open-loop system in \eqref{LTI} is unstable, the inequality condition in Lemma \ref{lem:PreviousSensitiveLemma} is an event of probability zero, that is, it almost never holds. This is because when $\rho(\A)>1$, $\|\A\A^\top\|_2>\rho(\A)>1$. This implies that there does not exist $\eta\geq 1$ such that $ \Psi=\A\A^\top \preceq  \left( 1-\frac{1}{\eta}\right)\I_n$. Therefore, Lemma~\ref{lem:PreviousSensitiveLemma} essentially says that, when the open-loop system is stable, control gains being zero matrix is stabilizing, as expected. The proof of Corollary \ref{coro:PreviousGuaranteeViolation} is  provided in Section \ref{proof:Theorem2}.

\subsection{Proofs of Theorem~\ref{th:CE_inconsistency} and Corollary~\ref{coro:PreviousGuaranteeViolation}} \label{proof:Theorem2}

We start by defining a change of variables to obtain a noise-free LTI system by treating the noise as an additional input. To this end, consider
 \begin{equation}
     \x_{t+1} = \bar{\A} \x_t + \bar{\B} \mathbf{v}_t,
     \label{GeneralNoiselessLTI}
 \end{equation}
 where the system state $\x_t \in \Rb^n$, the input $\bv_t= [\bu_t\quad \bw_t]^\top \in \Rb^{m+n}$, $\bar{\A}=\A, $ and $ \bar{\B}=[\B \quad \I_n]$. Therefore, we can write $\X_1$ defined in \eqref{data_matrix} as
\begin{equation}
    \label{eq:relationDataMatrices}
    \X_1 = \bar{\A} \X_0 + \bar{\B}   \left[\begin{array}{c }
 \U_0   \\
 \W_0 
\end{array}\right]  .
\end{equation}
We denote the input trajectory matrix as $\V_0=\left[\begin{array}{cc}
 \U_0^\top & \W_0^\top
\end{array}\right]^\top$. We use the fundamental lemma, which first appeared in Corollary 2 of \cite{willems2005note}, in our proof. Its detailed proof can be found in \cite{van2020willems}. 
 \begin{lemma}[Fundamental Lemma for Input-State Data, Theorem 1 in \cite{van2020willems}] Let
       $T \in \Nbp$. Consider the state and input trajectory matrices $\X_0$ and $\V_0   $  of \eqref{GeneralNoiselessLTI}.  Assume $\V_0 $ is persistently exciting of order $1+n$  in the sense of Definition \ref{persistencyOfExcitation}. Then,
$$
\operatorname{rank}\left(\left[\begin{array}{l}
\X_0    \\
\V_0  
\end{array}\right]\right)= 2n+m.
$$
\label{FundamentalLemma}
 \end{lemma}
 The rest of our proof depends on the properties of the following combined data matrix
\begin{equation}
\label{combined_data_matrix}
\D_T := \left[\begin{array}{c }
\X_0  \\
\U_0  \\
\X_1
\end{array}\right].
\end{equation}
Next, based on Lemma \ref{FundamentalLemma}, we show that the combined data matrix $\D_T$ is full row-rank with probability one.
\begin{lemma}
\label{FirstFullRowRank}
Consider the LTI system in \eqref{LTI}.
Assume {$\sigma_w>0$}. When $T\geq (m+n)(n+1)+n$, the combined data matrix $\D_T$ in \eqref{combined_data_matrix} satisfies 
    \begin{equation}
    \label{Full_rank_data_matrix}
    \Pb_T \left( \operatorname{rank}\left( \D_T \right) = 2n+m\right) =1 .\end{equation}
    
\end{lemma}

\begin{proof}
By \eqref{eq:relationDataMatrices}, we get
    \begin{equation}
      \D_T =  \bP_1  \bar{\D}_T,
\label{TwoDataMatrices}
      \end{equation}
      where $\bP_1 := \left[\begin{array}{ccc}
\I_n & \mathbf{0}_{n\times m} & \mathbf{0}_{n\times n} \\
\mathbf{0}_{m\times n} & \I_m & \mathbf{0}_{m\times n} \\
\A &\B &\I_n  
\end{array}\right]$ is a nonsingular matrix and $\bar{\D}_T := 
\left[\begin{array}{ccc}
\X_0^\top &
 \U_0^\top &
 \W_0^\top
\end{array}\right]^\top$.
$\bar{\D}_T$ can be considered as an input-state data matrix of the noise-free LTI system \eqref{GeneralNoiselessLTI}.
Note that $(\bar{\A},\bar{\B})$ defined in \eqref{GeneralNoiselessLTI} is controllable for any $\A$ and $\B$. Moreover, as we show in Lemma~\ref{persistentinputnoise} in the Appendix, with probability one, $\V_0$ is persistently exciting of order $n+1$ when $T\geq (m+n)(n+1)+n$. Hence, by Lemma \ref{FundamentalLemma}, we get that  
$\mathbb{P}_T\left(\operatorname{rank}\left(\bar{\mathbf{D}}_T \right)=2n+m\right)=1
$. Noting nonsingularity of $\bP_1$ completes the proof. 
\end{proof}

 The following lemma shows that any optimal solution of the CE DDD LQR problem \eqref{CEDDDLQRSDP} has to satisfy a linear equation, which yields a zero state feedback gain when there is noise, i.e., $\sigma_w>0$.
\begin{lemma}
Assume the combined data matrix $\D_T$ is full row rank, i.e., $\operatorname{rank}\left( \D_T \right) = 2n+m$. Consider the following underdetermined systems of equations: 
     \begin{equation}
         \D_T \Y=\left[\begin{array}{c }
\I_n  \\
 \mathbf{0}_{m\times n} \\
\mathbf{0}_{n\times n}
\end{array}\right].
         \label{OptimalSolutionEquation}
     \end{equation}
   Then, $\Y^*$ is an optimal solution of \eqref{CEDDDLQRSDP} if and only if $\Y^*$ satisfies \eqref{OptimalSolutionEquation}. 
     \label{GloballyOptimalOfNoiseless}
\end{lemma}
\begin{proof}
We note that since $\operatorname{rank}\left( \D_T \right) = 2n+m$ and $\Y\in\Rb^{T\times n}$, Equation \eqref{OptimalSolutionEquation} always has a solution. 

In the first part of the proof, we prove that any solution of
\eqref{OptimalSolutionEquation} is an optimal solution of \eqref{CEDDDLQRSDP}. 
By Schur complement, we can rewrite the first constraint in \eqref{CEDDDLQRSDP} as 
$\X_{0} \Y - \I_n \succeq 0
$ and $\X_{0} \Y - \I_n - \X_{1}  \Y \left( \X_{0} \Y\right)^{-1}  \Y^{\top} \X_{1}^{\top}\succeq 0$. Similarly, the second constraint can be written  as $\X \succeq 0$ and $\X - \sqrt{\R} \U_0 \Y \left( \X_0 \Y\right)^{-1}  (\sqrt{\R} \U_0 \Y )^{\top} \succeq 0.$ Since $\operatorname{trace}\left(\X \right)$ is being minimized, by the last inequality, the following will be satisfied at the optimal solution: 
$$\operatorname{trace}\left(\X \right) =\operatorname{trace}\left( \sqrt{\R} \U_0 \Y \left( \X_0 \Y\right)^{-1}  \left(\sqrt{\R} \U_0 \Y \right)^{\top} \right).$$ Hence, we can remove $\X$ from \eqref{CEDDDLQRSDP} to obtain an equivalent optimization problem:
 \begin{equation}
\begin{aligned} 
\underset{  \Y }{\operatorname{minimize}} &  \operatorname{trace}\left(\Q \X_0 \Y  \right)+\\
&\operatorname{trace}\left(  \sqrt{\R} \U_0 \Y \left( \X_0 \Y\right)^{-1}  \left(\sqrt{\R} \U_0 \Y \right)^{\top} \right) \\
\text { subject to } 
& \X_{0} \Y - \I_n \succeq 0\\
&\X_{0} \Y - \I_n - \X_{1}  \Y \left( \X_{0} \Y\right)^{-1}  \Y^{\top} \X_{1}^{\top}\succeq 0 ,
\end{aligned}
\label{CEDDDLQRNonconvex}
\end{equation}
where $\Y\in \Rb^{T \times n}$. For the simplicity of the notation, we use $O_{ce}\left( \Y\right)$ to denote the objective function of the above optimization problem.

Given $\X_{0} \Y - \I_n \succeq 0$, we see that for any feasible solution $\Y$ of \eqref{CEDDDLQRNonconvex}, we have $\operatorname{trace}\left(\Q \X_0 \Y  \right) \geq \operatorname{trace}\left(\Q   \right) $.
Moreover, 
    $   \sqrt{\R} \U_0 \Y \left( \X_0 \Y\right)^{-1}  (\sqrt{\R} \U_0 \Y )^{\top}\succeq 0$ implies that
    $O_{ce}\left( \Y\right) \geq \operatorname{trace}\left(\Q \right)$ for any feasible solution $\Y$ of \eqref{CEDDDLQRNonconvex}. Therefore, if there exists a feasible solution $\Y_{m}$ such that $ O_{ce}\left( \Y_{m}\right)= \operatorname{trace}\left(\Q \right)$, then $\Y_{m}$ is an optimal solution of \eqref{CEDDDLQRNonconvex}. Take a solution $\Y^*$ of \eqref{OptimalSolutionEquation}, then
      two matrix inequality constraints of \eqref{CEDDDLQRNonconvex} are satisfied with equality and we also have
      $O_{ce}\left( \Y^*\right) = \operatorname{trace}\left(\Q \X_0\Y^* \right)= \operatorname{trace}\left(\Q \I_n \right)= \operatorname{trace}\left(\Q \right).$ Therefore, any solution of \eqref{OptimalSolutionEquation} is an optimal solution of \eqref{CEDDDLQRNonconvex} and the optimal objective value of   \eqref{CEDDDLQRNonconvex} is $\operatorname{trace}\left(\Q \right)$.
      
 In the second part of the proof, we will prove that any optimal solution $\Y_{ce}^*$ of   \eqref{CEDDDLQRNonconvex} must satisfy \eqref{OptimalSolutionEquation}. Assume by contradiction that  $\X_{0} \Y^*_{ce} \neq \I_n$. According to the first matrix inequality constraint in \eqref{CEDDDLQRNonconvex},  $\X_{0} \Y^*_{ce} \neq \I_n$ implies that there exists $\C_1 \succeq 0$ with $\C_1\neq 0$ such that $\X_0 \Y^*_{ce} = \C_1 + \I_n$. Hence, 
\begin{equation}
\label{ObjectiveIneq1}
\begin{aligned}
    O_{ce}(\Y^*_{ce}) \geq \operatorname{trace}\left(\Q \left(\C_1+\I_n \right) \right).
    \end{aligned}
\end{equation}
Since $\Q \succ 0$ and $\C_1 \succeq 0$ with $\C_1 \neq 0$ , we have that $\operatorname{trace}\left(\Q  \C_1 \right)$ is strictly positive. Then from \eqref{ObjectiveIneq1}, we have that $ O_{ce}(\Y^*_{ce}) >  \operatorname{trace}\left(\Q \right),$ which contradicts with the fact that the optimal value of \eqref{CEDDDLQRNonconvex} is $\operatorname{trace}\left(\Q \right)$.
Therefore, we conclude that  $\X_0 \Y^*_{ce} = \I_n$ must hold. Based on the second constraint in \eqref{CEDDDLQRNonconvex}, $\X_0 \Y^*_{ce} = \I_n$ implies that $\X_1 \Y^*_{ce} = \mathbf{0}_{n\times n}$ must hold. According to the objective function definition in \eqref{CEDDDLQRNonconvex} and the fact that the optimal objective value of \eqref{CEDDDLQRNonconvex} is $\operatorname{trace}\left(\Q \right)$, we necessarily have $\U_0 \Y^*_{ce} = \mathbf{0}_{m \times n}$.
In conclusion, any optimal solution of \eqref{CEDDDLQRNonconvex} satisfies \eqref{OptimalSolutionEquation}.
\end{proof}



When $\sigma_w =0,$ the noise matrix $\W_0$ in $\bar{\D}_T$ becomes a zero matrix. 
Therefore, Equation \eqref{OptimalSolutionEquation} reduces to 

\begin{equation}
\begin{aligned}\left[\begin{array}{c }
\X_0  \\
\U_0 \\
\A\X_0+\B\U_0
\end{array}\right] \Y = \left[\begin{array}{c }
\I_n  \\
 \mathbf{0}_{m\times n} \\
\mathbf{0}_{n\times n}
\end{array}\right].
\end{aligned}
\label{noiseless_optimal_equation}
\end{equation}
The above equation has a solution only if $\A=\mathbf{0}_{n\times n}.$ Therefore, in the noiseless case, a trivial solution cannot be constructed by \eqref{OptimalSolutionEquation} in general.

On the other hand, when there is noise, last block row of $\D_T$ will be perturbed by noise and for many noise distributions $\D_T$ will be full rank, leading to a zero gain as the output of the optimization problem~\eqref{CEDDDLQRSDP}. This is the discontinuity in the solution that was alluded to in Fig.~\ref{diagram}.

We also note that
if the input $\bu_t$ is {i.i.d.} and from a continuous distribution; and noise $\bw_t$ comes from any continuous distribution, it is {i.i.d.} and independent of the input, the combined data matrix $\D_T$ will have full row-rank with probability one. So, a trivial controller gain will be obtained for any such input/noise distribution as well.

Next, we give the proof of Theorem \ref{th:CE_inconsistency}. \label{ProofOfTheorem1}
\begin{proof} (of Theorem \ref{th:CE_inconsistency})
   By the definition of $\K_{ce}$, Lemma \ref{FirstFullRowRank}, and Lemma \ref{GloballyOptimalOfNoiseless}, any optimal solution of \eqref{CEDDDLQRSDP} yields the unique $\K_{ce}=\mathbf{0}_{m\times n}$ when $T\geq (m+n)(n+1)+n$, which holds true with probability one. This completes the proof.
\end{proof}

Now, we can prove Corollary~\ref{coro:PreviousGuaranteeViolation}.

\begin{proof} (of Corollary~\ref{coro:PreviousGuaranteeViolation})
  Under the assumptions of Theorem \ref{th:CE_inconsistency},  by Lemma \ref{FirstFullRowRank} and Lemma \ref{GloballyOptimalOfNoiseless}, we get that for an optimal solution $\Y_{ce}^*$ of CE DDD LQR, $$ \begin{cases}\X_0 \Y_{ce}^*=\I_n\\
  \U_0\Y_{ce}^*=\mathbf{0}_{m\times n} \\
  \X_1 \Y_{ce}^*=\mathbf{0}_{n\times n}\end{cases},$$ 
with probability $1$. Then based on \eqref{eq:relationDataMatrices}, this implies 
    $$
    \begin{aligned}
         \W_0 \Y_{ce}^*=\left( \X_1 - \A  \X_0 - \B  \U_0 \right)  \Y_{ce}^* = -\A,
    \end{aligned}
    $$ 
  with probability $1$.  Therefore, by the definition of $\Psi $ in Lemma \ref{lem:PreviousSensitiveLemma}, we conclude that $\Pb_T \left(\Psi = \A\A^\top\right)=1$.
\end{proof}

 \section{ Robustness-Promoting (RP) DDD LQR}
 \label{sec_RP_DDD_LQR}
 \subsection{ Inconsistency of RP DDD LQR}
 \label{InconsistencyrobustnessPromoting}
Next, we analyze the solution of RP DDD LQR problem in \eqref{RPDDDLQRSDP} as the data trajectory length $T$ goes to infinity. 
\begin{theorem}
Consider the LTI system in \eqref{LTI} and assume that {$\sigma_w>0$}. Let $ \Y^*_{rp}(T)$ be an arbitrary optimal solution of \eqref{RPDDDLQRSDP} for a given $T$ and let $\K_{rp}(T) =-\U_0(T) \Y^*_{rp}(T) \left(\X_0(T) \Y^*_{rp}(T) \right)^{-1}$ be the corresponding feedback control gain as defined in \eqref{eq:Krp}.
Then,
\begin{equation}
\K_{rp}(T) \overset{p}{\rightarrow} \mathbf{0}_{m\times n}.
\end{equation}
    \label{th:RP_inconsistency}
\end{theorem}

Theorem~\ref{th:RP_inconsistency} says that when the length of the data trajectory $T$  approaches infinity, the data-driven feedback gain based on RP DDD LQR in \eqref{RPDDDLQRSDP} asymptotically converges to $\mathbf{0}_{m\times n}$ in probability, similar to the non-asymptotic almost sure result for CE DDD LQR. Therefore, robustness promoting regularization does not completely eliminate the sensitivity to noise. The proof of this theorem can be found in Section \ref{proofofMainResultForRobustnessPromotion}.

\subsection{Proof of Theorem~\ref{th:RP_inconsistency}}

\label{proofofMainResultForRobustnessPromotion}

Similar to the proof structure in Section \ref{proof:Theorem2}, we first do a change of variables to define a noise-free system with isotropic random inputs:
\begin{equation}
     \x_{t+1} = \tilde{\A} \x_t + \tilde{\B} \z_t,
     \label{NoiselessLTI_isometric_input}
 \end{equation}
  where the system state $\x_t \in \Rb^n$, the input $\z_t=[ \bu_t\quad \frac{\sigma_u}{\sigma_w } \bw_t]^\top \overset{i.i.d.}{\sim} \mathcal{N}\left(0,  {\sigma_u^2} \I_{m+n}\right) \in\Rb^{m+n}$, $\tilde{\A}:=\A$, and $\tilde{\B}:= [ \B \quad \frac{\sigma_w}{\sigma_u} \I_n]$.  Hence, the trajectory matrices in \eqref{data_matrix} are related by
\begin{equation}
    \label{eq:relationDataMatrices_isotropic}
    \X_1 = \tilde{\A} \X_0 + \tilde{\B}   \left[\begin{array}{c }
   \U_0   \\
 \frac{\sigma_u}{\sigma_w} \W_0 
\end{array}\right].
\end{equation}
Let the input trajectory matrix in \eqref{eq:relationDataMatrices_isotropic} be denoted as $\Z_0=\left[\begin{array}{cc}
 \U_0^\top &
 \frac{\sigma_u}{\sigma_w} \W_0^\top
\end{array}\right]^\top$. The following lemma, which combines Lemma 1 and Theorem 3 in \cite{coulson2022quantitative}, provides a quantitative measure of persistency of excitation in Lemma~\ref{FundamentalLemma}.
\begin{lemma} \label{LowerLeastSingualeDataMatrices} 
 Consider the state and input trajectories $\X_0$ and $\Z_0   $  of \eqref{NoiselessLTI_isometric_input}. 
 Then, there exists a positive constant $\rho$, which is determined only by $(\tilde{\A}, \tilde{\B})$,
   such that
        \begin{equation}\label{eq:quant_FL}
        \begin{aligned}
&\underline{\sigma}\left( \left[\begin{array}{c}
\X_0  \\
\Z_0
\end{array}\right] \right) \geq  \underline{\sigma} \left(\mathcal{H}_{n+1}\left(\Z_0\right)\right) \frac{\rho}{\sqrt{n+1}}.
\end{aligned}
\end{equation}
   
\end{lemma}

We refer the reader to \cite{coulson2022quantitative} for computation of the constant $\rho$ as a function of $(\tilde{\A}, \tilde{\B})$. In the stochastic setting, we can further bound the minimum singular value of the Hankel matrix in \eqref{eq:quant_FL} with high probability as follows.

\begin{lemma}[Lemma C.2 in \cite{oymak2019non}] 
\label{LowerBoundOfGaussianInput} Consider the input matrix $\Z_{0}$ of \eqref{NoiselessLTI_isometric_input}, with $\bz_t \overset{i.i.d.}{\sim} \mathcal{N}\left(0, \sigma^2_z \I_{m+n}\right)$ and  {$\sigma_z>0$}. 
If the trajectory length satisfies {$T \geq c \Lambda(m,n,T)$,  for $\Lambda(m,n,T):=(n+1) (m+n) \log^2(2 (n+1) (m+n)) \log^2(2 T (m+n))$} and a sufficiently large constant $c>0$, then
we have 
\begin{equation}
\label{eq_least_singular_Hankel}
\Pb_T \left(  \underline{\sigma}\left(  \mathcal{H}_{n+1}\left(\Z_0(T)\right)\right) \geq \frac{\sqrt{T-n} \sigma_z}{\sqrt{2}}  \right) \geq 1-\epsilon_T,
\end{equation}
where $\epsilon_T := (2T(m+n))^{-\log^2(2(n+1) (m+n) ) \log(2 T(m+n) )}$ and $\Pb_T$ is with respect to the the randomness of $\Z_{0}$.
\end{lemma}

 
Hence, since $\epsilon_T$ is a decreasing function of $T$, the quantitative measure of persistency of excitation in Lemma~\ref{LowerLeastSingualeDataMatrices} grows with $\sqrt{T}$ when the system is driven by isometric Gaussian input. Therefore, based on Lemma \ref{LowerLeastSingualeDataMatrices} and Lemma \ref{LowerBoundOfGaussianInput}, we can show that the least singular value of the combined data matrix $\D_T$ grows with $\sqrt{T}$.

\begin{lemma}
\label{LowerBoundOfGaussianData}
Consider the LTI system in \eqref{LTI} and its combined data matrix $\D_T$ in \eqref{combined_data_matrix} with  {$\sigma_w>0$}. When the length of trajectory obeys $T \geq c\Lambda(m,n,T)$ for a sufficiently large constant $c>0$, there exists a positive constant $\rho$, which is determined only by $(\tilde{\A}, \tilde{\B})$, such that
\begin{equation}
\Pb_T \left(  \underline{\sigma}\left( \D_T\right) \geq  \underline{\sigma}\left(\bP_2 \right)   \frac{\sqrt{T-n}\rho \sigma_u }{\sqrt{2(n+1)} } \right) \geq 1-\epsilon_T,
\end{equation}
 where the probability $\Pb_T$ is with respect to the randomness of $x_0$, $\U_0(T)$, and $\W_0(T)$.
\end{lemma}

\begin{proof}
By \eqref{eq:relationDataMatrices_isotropic}, the combined data matrix $\D_T$ in \eqref{combined_data_matrix} can be represented  as
    \begin{equation}
      \D_T =  \bP_2  \tilde{\D}_T,
\label{TwoDataMatrices_isometric}
      \end{equation}
   where $\bP_2 := \left[\begin{array}{ccc}
\I_n & \mathbf{0}_{n\times m} & \mathbf{0}_{n\times n} \\
\mathbf{0}_{m\times n} & \I_m & \mathbf{0}_{m\times n} \\
\A &   \B & \frac{\sigma_w }{\sigma_u} \I_n  
\end{array}\right]$ and  $\tilde{\D}_T := 
\left[\begin{array}{ccc}
\X_0^\top &
 \U_0^\top &
 \frac{\sigma_u}{\sigma_w}\W_0^\top
\end{array}\right]^\top.$ $\tilde{\D}_T$ can seen as the input-state data matrix of the noise-free system \eqref{NoiselessLTI_isometric_input}. We also note that, by definition, $\bz_t$ has covariance $\sigma_u^2 \I_{m+n}$; 
and $(\tilde{\A},\tilde{ \B })$ given in \eqref{NoiselessLTI_isometric_input} is controllable for any $\A$ and $\B$. Therefore, combining Lemma \ref{LowerLeastSingualeDataMatrices} and Lemma \ref{LowerBoundOfGaussianInput},  when the length of trajectory obeys $T \geq c\Lambda(m,n,T)$ for sufficiently large constant $c>0$,  we have  
\begin{equation}
\begin{aligned}
&\Pb_T \left(   \underline{\sigma}\left( \tilde{\D}_T\right) \geq \frac{\sqrt{T-n}\rho \sigma_u }{\sqrt{2(n+1)}}\right)  \geq 1-\epsilon_T.
\label{SecondLeastSingularValue}
\end{aligned}
\end{equation}

Combining \eqref{TwoDataMatrices_isometric}, \eqref{SecondLeastSingularValue}, and the inequality $\underline{\sigma}\left(\bP_2  \tilde{\D}_T \right) \geq\underline{\sigma}\left(\bP_2    \right) \underline{\sigma}\left( \tilde{\D}_T \right) $,  we complete the proof.
\end{proof}

Next, we give the proof of Theorem \ref{th:RP_inconsistency}.
\begin{proof} (of Theorem \ref{th:RP_inconsistency})
The proof of this theorem is structured as follows: \textbf{Step 1:} A lower bound for the  objective value of any feasible solution of RP DDD LQR problem \eqref{RPDDDLQRSDP} is established. \textbf{Step 2:} An upper bound for the optimal objective value of \eqref{RPDDDLQRSDP} is derived.
\textbf{Step 3:} An upper bound of $\overline{\sigma}\left(\U_0  \Y_{rp}^*(T)\right)$ is determined.
\textbf{Step 4:} It is further shown that, as the trajectory length $T$ approaches infinity, $\K_{rp}(T)$ converges to the zero matrix  in probability. 

\textbf{Step 1:} Following a similar approach used to obtain \eqref{CEDDDLQRNonconvex}, and utilizing the Schur complement and removing $\X$ and $\Sb$, Problem \eqref{RPDDDLQRSDP} can be reformulated as the following equivalent problem:
 \begin{equation}
\begin{array}{cl}
\underset{  \Y }{\operatorname{minimize}} & \operatorname{trace}\left(\Q \X_0 \Y  \right)+\\
&\operatorname{trace}\left(  \sqrt{\R} \U_0 \Y \left( \X_0 \Y\right)^{-1}  \left(\sqrt{\R} \U_0 \Y \right)^{\top} \right)\\
&+ \eta \operatorname{trace}\left(   \Y \left( \X_0 \Y\right)^{-1}   \Y^{\top} \right) \\
\text { subject to } 
& \X_{0} \Y - \I_n \succeq 0\\
&\X_{0} \Y - \I_n - \X_{1}  \Y \left( \X_{0} \Y\right)^{-1}  \Y^{\top} \X_{1}^{\top}\succeq 0 ,
\end{array}
\label{RPDDDLQRNonconvex}
\end{equation} 
where $\Y\in \Rb^{T \times n}$. We use $O_{rp}\left( \Y\right)$ to denote the objective function and $O_{rp}^*(T)$ for the optimal objective value of the above optimization problem. 
 For any feasible solution $\Y_f$ of \eqref{RPDDDLQRNonconvex}, we have 
 $\X_{0} \Y_f \succeq \I_n $. Therefore,  
 $$ \sqrt{\R} \U_0 \Y_f \left( \X_0 \Y_f \right)^{-1}  \left(\sqrt{\R} \U_0 \Y_f  \right)^{\top} \succeq 0,$$ and $   \Y_f \left( \X_0 \Y_f  \right)^{-1}   \Y_f^{\top}  \succeq 0$. Then for any feasible solution $\Y_f$,
\begin{equation}
\label{upper_bound_trace_QX0Y}
    \operatorname{trace}\left(\Q \X_0 \Y_f  \right) \leq     O_{rp}\left( \Y_f\right).
\end{equation}

\textbf{Step 2:} Let $\Y_{n}$ denote the minimum {(Frobenius)} norm solution of the  underdetermined equation \eqref{OptimalSolutionEquation}. By Lemma~\ref{FirstFullRowRank}, when $T>(m+n)(n+1)+n$, $\D_T$ is full row rank with probability one, hence the minimum norm solution exists and is equal to:
\begin{equation}
\begin{aligned}
\Y_n &= \D_T^\top \left( \D_T \D_T^\top\right)^{-1} \E,
\end{aligned}
\label{minimumNormSolution}
\end{equation}
where $\E$ denotes the right-hand side of \eqref{OptimalSolutionEquation}.

It can be seen that $\Y_{n}$ satisfies the constraints in \eqref{RPDDDLQRNonconvex}, hence it is a feasible solution. 
Therefore, we have $
    O_{rp}^*(T) \leq O_{rp}\left(\Y_{n} \right).$ Furthermore, we have
\begin{equation}
\begin{aligned}
    &O_{rp}^*(T) \leq O_{rp}\left(\Y_{n} \right)\\
    & = \operatorname{trace}\left(  \sqrt{\R} \U_0 \Y_n \left( \X_0 \Y_{n}\right)^{-1}  \left(\sqrt{\R} \U_0 \Y_{n}  \right)^{\top} \right)\\
    &\quad + \operatorname{trace}\left(\Q \X_0 \Y_{n}  \right) + \eta\operatorname{trace}\left(   \Y_{n} \left( \X_0 \Y_{n}\right)^{-1}   \Y_{n}^{\top} \right) \\
    &\overset{(a)}{=}\operatorname{trace}\left( \Q\right) + \eta \operatorname{trace} \left(\Y_n^\top \Y_n\right)\\
    &\overset{(b)}{=} \operatorname{trace}\left( \Q\right) + \eta \operatorname{trace} \left(\E^\top  \left(  \D_T \D_T^\top \right)^{-1}\E\right)\\
    &\overset{(c)}{\leq}  \operatorname{trace}\left( \Q\right) + \eta (2n+m)  \left\| \E^\top  \left(  \D_T \D_T^\top \right)^{-1}\E\right\|_2\\
& \leq \operatorname{trace}\left( \Q\right) + \eta(2n+m) \left\|\E  \right\|_2 \left\| \left(  \D_T \D_T^\top \right)^{-1}\right\|_2 \left\|\E\right\|_2\\
&\overset{(d)}{\leq} \operatorname{trace}\left( \Q\right) + \frac{\eta(2n+m)}{ \underline{\sigma}\left( \D_T \D_T^\top  \right)},
\end{aligned}
\label{upperObjectiveValue}
\end{equation}
where the equality $(a)$ is because $\X_0\Y_n=\I_n$ and $\U_0\Y_n=\mathbf{0}_{m\times n}$ in \eqref{OptimalSolutionEquation}, the   equality $(b)$ is due to \eqref{minimumNormSolution}, the inequality $(c)$ is from the fact that $\operatorname{trace}(\M) \leq n \|\M\|_2$ for a square matrix $\M\in \mathbb{R}^{n \times n}$, and the inequality $(d)$ is based on the facts that $\|\E\|_2 = 1$ and $ \|  (  \D_T \D_T^\top  )^{-1} \|_2 \leq \frac{1}{ \underline{\sigma} ( \D_T \D_T^\top   )}$.

\textbf{Step 3:} Consider an arbitrary optimal solution $\Y^*_{rp}(T)$ of \eqref{RPDDDLQRSDP}. Combining \eqref{upper_bound_trace_QX0Y} and \eqref{upperObjectiveValue}  gives
\begin{equation}
\label{lower_bound_QX0Y}
\operatorname{trace}\left(\Q \X_0 \Y_{rp}^*(T)  \right) \leq \operatorname{trace}\left( \Q\right) + \frac{\eta(2n+m)}{ \underline{\sigma}\left( \D_T \D_T^\top  \right)},\end{equation}
from which we obtain 
\begin{equation}
\label{upper_sigma_x0Y}
  0 \leq  \overline{\sigma}\left(\X_0  \Y_{rp}^*(T) \right) \leq \frac{\operatorname{trace}(\Q)}{\underline{\sigma}(\Q)} + \frac{\eta(2n+m)}{\underline{\sigma}(\Q) \underline{\sigma}\left( \D_T \D_T^\top  \right)},
\end{equation}
where the second inequality is based on the fact $\operatorname{trace}(\M_1\M_2)\geq \operatorname{trace}(\M_1)\underline{\sigma}(\M_2)\geq \overline{\sigma}(\M_1)\underline{\sigma}(\M_2)$, for $\M_1,\M_2\succeq0$.

Thanks to $ \X_{0} \Y_{rp}^*(T) - \I_n \succeq 0$, we have
\begin{equation}
\label{upper_U0Y1}
\begin{aligned}
&\operatorname{trace}\left(\Q   \right) + \\
&\operatorname{trace}\left(  \sqrt{\R} \U_0 \Y \left( \X_0  \Y_{rp}^*(T) \right)^{-1}  \left(\sqrt{\R} \U_0  \Y_{rp}^*(T)  \right)^{\top} \right) \\
&\leq O^*_{rp}(T) \leq \operatorname{trace}\left(\Q   \right)+\frac{\eta(2n+m)}{ \underline{\sigma}\left( \D_T \D_T^\top \right)},
\end{aligned}
\end{equation}
where the last inequality follows from \eqref{upperObjectiveValue}.
Hence we have  
\begin{equation}
\label{upper_trace_U0Y_1}
\begin{aligned}
&0 \leq \operatorname{trace}\left(   \left(\sqrt{\R} \U_0  \Y_{rp}^*(T)  \right)^{\top} \sqrt{\R} \U_0  \Y_{rp}^*(T)\right) \\
& \leq \frac{\eta(2n+m)}{ \underline{\sigma}\left( \D_T \D_T^\top \right)} \overline{\sigma}\left(\X_0  \Y_{rp}^*(T) \right).
\end{aligned}
\end{equation}

Finally, combining \eqref{upper_sigma_x0Y} and \eqref{upper_trace_U0Y_1}    yields
\begin{equation}
\label{upper_sigma_UY}
\begin{aligned} &\left(\overline{\sigma}\left(\U_0  \Y_{rp}^*(T)\right)\right)^2 \leq \\
 &\frac{\eta(2n+m)}{\underline{\sigma}(\R) \underline{\sigma}\left( \D_T \D_T^\top \right)}  \left[\frac{\operatorname{trace}(\Q)}{\underline{\sigma}(\Q)} + \frac{\eta(2n+m)}{\underline{\sigma}(\Q) \underline{\sigma}\left( \D_T \D_T^\top  \right)}\right].
 \end{aligned}
\end{equation}

\textbf{Step 4:} As $\X_0 \Y^*_{rp}(T) \succeq \I_n $, we have
\begin{equation}
\label{lower_sigma_XY}
    \underline{\sigma}(\X_0 \Y^*_{rp}(T)) \geq 1 \geq \overline{\sigma}\left((\X_0 \Y^*_{rp}(T)\right)^{-1}).
\end{equation}
 
  Because  $\K_{rp}(T) = -\U_0 \Y^*_{rp}(T) \left(\X_0 \Y^*_{rp}(T) \right)^{-1},$  
  then \eqref{upper_sigma_UY} and \eqref{lower_sigma_XY} imply
  \begin{equation}
  \label{upper_K_rp}
  \begin{aligned}
    & \|\K_{rp}(T)\|^2_2  \\
    &\leq  \left( \overline{\sigma}\left(\U_0  \Y_{rp}^*(T)\right) \right)^2 
    \\
    &\leq  \frac{\eta(2n+m)}{\underline{\sigma}(\R) \underline{\sigma}\left( \D_T \D_T^\top \right)}  \left[\frac{\operatorname{trace}(\Q)}{\underline{\sigma}(\Q)} + \frac{\eta(2n+m)}{\underline{\sigma}(\Q) \underline{\sigma}\left( \D_T \D_T^\top  \right)}\right]  .
    \end{aligned}
\end{equation}

 Combining \eqref{upper_K_rp} and Lemma \ref{LowerBoundOfGaussianData}, when the length of the trajectory obeys $T \geq c\Lambda(m,n,T)$ for a sufficiently large constant $c>0$, we have that 

\begin{equation}
\label{upper_K_rp_prob}
\begin{aligned}
& \Pb_T \left( \|\K_{rp}(T)\|^2_2 \leq \frac{ C }{     T-n  }\left[\frac{\operatorname{trace}(\Q)}{\underline{\sigma}(\Q)}+
  \frac{ C }{    T-n } \right]\right)\\
  &\geq 1-\epsilon_T,
  \end{aligned}
\end{equation}
\noindent{where 
$C := \frac{2(n+1)   (2n+m) \eta }{\min\{\underline{\sigma}(\R) ,\underline{\sigma}(\Q)  \} \left(\underline{\sigma}\left(\bP_2 \right)\right)^2   \rho^2 \sigma_u^2 }$ 
contains the terms that are independent of the horizon $T$. }

Fix any $\delta >0$. Let $T^*  $ be the minimum $T$ such that  $\frac{ C }{    T-n  }\left[\frac{\operatorname{trace}(\Q)}{\underline{\sigma}(\Q)}+
  \frac{ C }{     T-n   } \right] \leq \delta$ and $\bar{T}$ be the minimum $T$ such that  $T \geq c\Lambda(m,n,T)$. Then, for all $T\geq \max(T^*, \bar{T})$, we have
\begin{equation}
\label{bound_w}
\Pb_T \left( \|\K_{rp}(T)\|^2_2 > \delta
  \right) < \epsilon_T.
\end{equation}
This implies, for all $\delta>0$, 
\begin{equation}
  \lim_{T\rightarrow \infty} \Pb_T  \left( \|\K_{rp}(T)\|^2_2> \delta \right)=0,\end{equation}
that is, 
$\|\K_{rp}(T)\|^2_2  \overset{p}{\rightarrow} 0.$ Hence,
$
    \K_{rp}(T)\overset{p}{\rightarrow}\mathbf{0}_{m\times n} .$
 \end{proof}

We note that the analysis in this section assumes a fixed regularization constant $\eta$ for all horizons $T$. However, the bound in~\eqref{upperObjectiveValue} suggests that scaling $\eta$ with $T$ can prevent the gain from collapsing to zero. This is investigated in the experiments in Section \ref{rp_ddd_lqr_in_eta}.
  
 \section{Numerical Experiments}
 \label{sec:num_ex}
In this section, we present three numerical experiments to validate our main theoretical results. The experiments include CE, RP with a fixed regularization parameter, and RP with an increasing regularization parameter. These experiments are conducted on a second-order single-input system, described by the dynamics in \eqref{LTI}, with the following system matrices:
 \begin{equation}
 \A = \left[\begin{array}{cc}
 0.8878 & 0.2232  \\
0.3491 & 0.3726  
\end{array}\right]\;\text{and} \;\B = \left[\begin{array}{c}
-0.6808  \\
0.3726   
\end{array}\right],
\label{sys_exp}
\end{equation}
where the spectral radius of $\A$ is $1.01$, indicating that this LTI system is open-loop unstable. The initial state $ \x_0  = \mathbf{0}_{2\times1}$. For the parameter matrices $(\Q,\R)$ in LQR problem, we set $\Q=\I_2$ and $\R=1$, yielding a true LQR gain:
\begin{equation}
 \K_{\lqr}=\left[\begin{array}{cc}
-0.7112 & -0.2046
\end{array}\right].
\label{K_lqr_exp}
\end{equation}

\subsection{CE DDD LQR} 
In this part, we validate the theoretical results of CE DDD LQR in Theorem \ref{th:CE_inconsistency} and compare it with the noiseless case in Theorem \ref{thm:NoiselessIsGood}. We consider the specific system in \eqref{sys_exp}. The data matrices are collected offline by exciting the system with random inputs,  sampled as $\bu_t \overset{i.i.d.}{\sim} \mathcal{N}(0, \sigma_{u}^2)$, with $\sigma_u^2=1$. The trajectory length is set to $T=50$. The implementation of CE DDD LQR in \eqref{CEDDDLQRSDP} is performed using YALMIP \cite{Lofberg2004} with MOSEK in MATLAB. As expected, for the noiseless case ($ \sigma_{w}^2=0$), the LQR gain estimate computed by CE DDD LQR is 
$\K_{ce}= \left[\begin{array}{cc}
-0.7112 & -0.2046
\end{array}\right],$
which matches the true LQR gain in \eqref{K_lqr_exp}. In contrast, for a small noise level ($ \sigma_{w}^2=0.00001$), the LQR gain estimate obtained from DDD LQR  is
$\K_{ce}= \left[\begin{array}{cc}
0 & 0
\end{array}\right].$
This estimate results in an unstable closed-loop system for the system in \eqref{sys_exp}, highlighting that CE DDD LQR is highly sensitive to noise, even at very low levels. This observation corroborates the validity of Theorem \ref{th:CE_inconsistency}.

\subsection{RP DDD LQR with a Fixed Regularization Parameter}
\label{exp_rp_ddd_lqr_fixed_eta}
In this section, we verify Theorem \ref{th:RP_inconsistency} for RP DDD LQR in \eqref{RPDDDLQRSDP} with a fixed regularization parameter $\eta=1$. Similar to the previous section for CE DDD LQR, we consider the same specific system dynamics as described in \eqref{sys_exp}, with the inputs in the data matrices unchanged. The spectral norm of the noise covariance matrix is set to $ \sigma_{w}^2=1$. The results are illustrated in Fig. \ref{fig_gain_change_fixed_eta} and 
Fig. \ref{fig_optimization_variable_change}. For each horizon $T$, we run $10$ independent experiments to compute the mean and the variance of each respective data point.  
From Fig. \ref{fig_gain_change_fixed_eta}, we have that when horizon $T$ increases, $\|\K_{rp}\|$ with $\sigma_w >0$ approaches zero, by which Theorem \ref{th:RP_inconsistency} is verified. 
Interestingly, when the data is generated from a system not subject to noise, 
with a fixed regularization parameter ($\eta=1$), $\|\K_{rp}\|$ converges to $\|\K_{\lqr}\|$ as the trajectory length $T$ increases
as shown in Fig. \ref{fig_gain_change_fixed_eta} (indeed, $\K_{rp}$ converges to $\K_{\lqr}$). Additionally, based on Fig. \ref{fig_optimization_variable_change}, we see that as $T$ increases, $\|\X_0 \Y_{rp}^*\|$ approaches $1$,  $\|\U_0  \Y_{rp}^*\|$ approaches $0$, and $\|\K_{rp} \|$ with $\sigma_w =0$ approaches $0$. In fact, it is possible to show that as $T$ approaches infinity, any optimal solution of RP DDD LQR converges in probability to a solution of equation \eqref{OptimalSolutionEquation}. 

\subsection{RP DDD LQR with an Increasing Regularization Parameter}
\label{rp_ddd_lqr_in_eta}

This section investigates what happens if the regularization parameter $\eta$ is increased with increasing horizon length $T$. The bound in~\eqref{upperObjectiveValue} suggests that if $\eta$ increases with the data trajectory length $T$, the optimal objective value of \eqref{RPDDDLQRSDP} may not converge to $\operatorname{trace}(\Q)$, thereby resulting in a nontrivial feedback gain. Motivated by this intuition,  we set $\eta=10T$. As shown in Fig. \ref{fig_gain_change_in_eta}, as the data trajectory length $T$ increases, $\|\K_{rp}\|$ with $\sigma_w >0$ does not approach zero. This observation highlights the potential of increasing $\eta$ systematically to obtain more reliable solutions, a direction that might be worthwhile to theoretically investigate.

\begin{figure}
  \centering
  \includegraphics[scale=0.38]{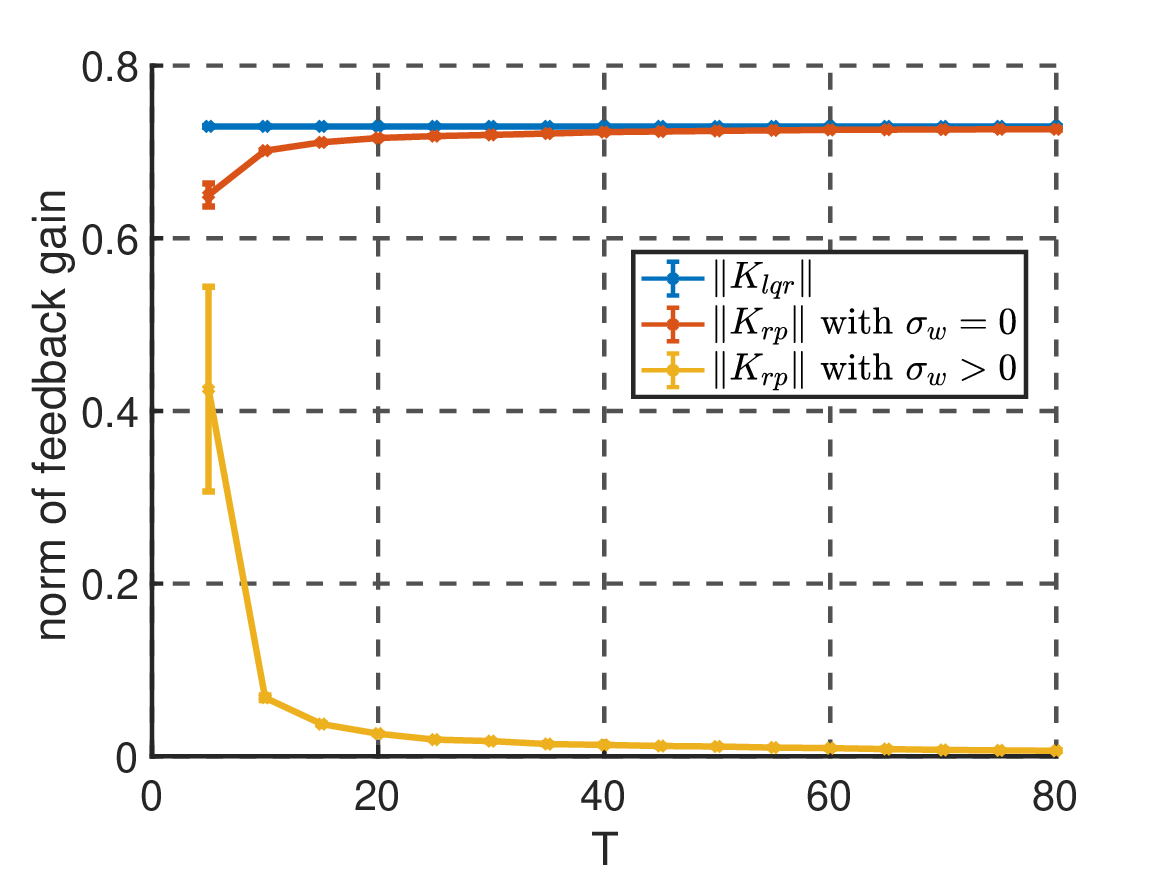}
  \caption{The $x$-axis is the length of the trajectory and the $y$-axis is the spectral norms of the feedback gains. The blue one is the spectral norm of the true LQR gain in \eqref{K_lqr_exp}, the orange one is the spectral norm of the feedback gain $\K_{rp}$ by RP DDD LQR without noise, and the yellow one is the spectral norm of the feedback gain estimate of  $\K_{rp}$ by RP DDD LQR with noise with a fixed $T$. The regularization parameter $\eta$ in RP DDD LQR is fixed for all $T$.}
  \label{fig_gain_change_fixed_eta}
\vspace{-0.3cm}
\end{figure}


\begin{figure}
  \centering
  \includegraphics[scale=0.38]{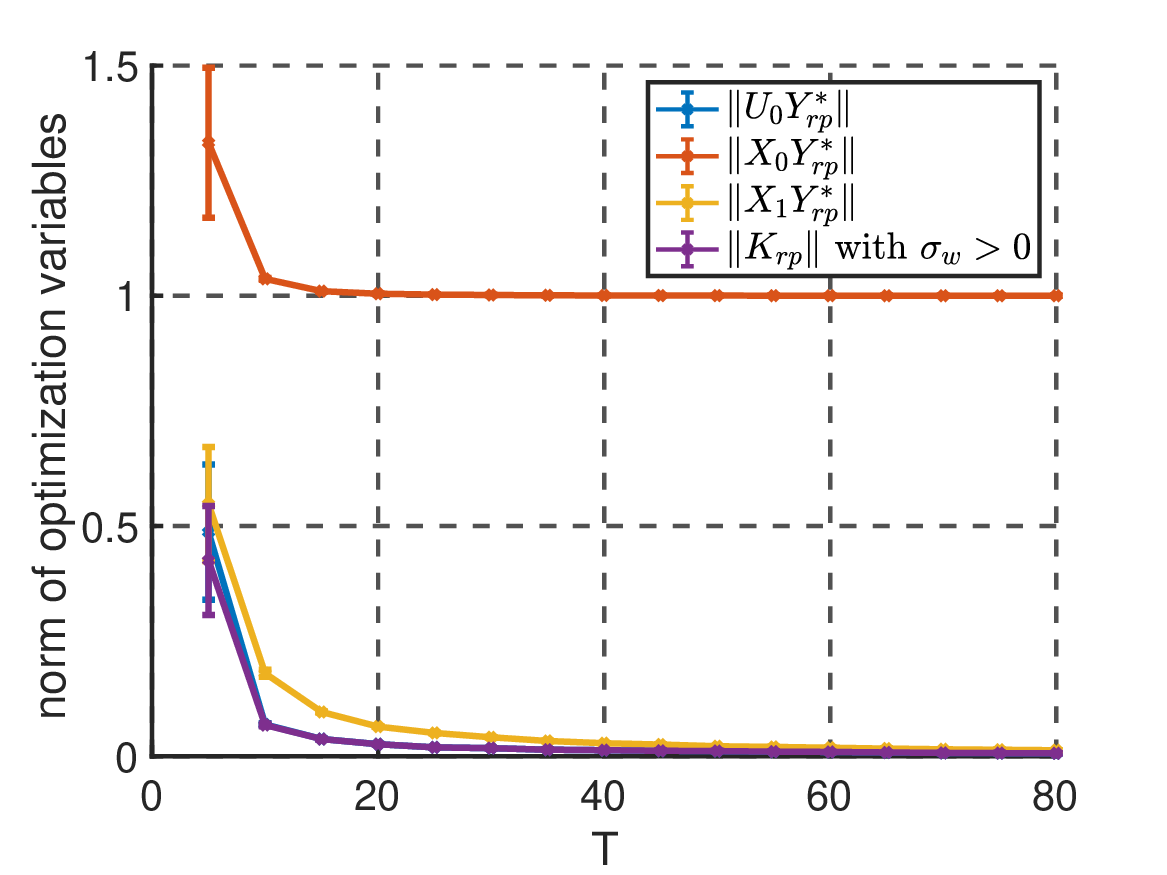}
  \caption{The $y$-axis is the norm of some optimal optimization variables. The blue one is the spectral norm of $\U_0  \Y_{rp}^*$, the orange one is the spectral norm of $\X_0 \Y_{rp}^*$, the yellow one is the spectral norm of $\X_1 \Y_{rp}^*$,  where $\Y_{rp}^*$ is an optimal solution of RP DDD LQR in \eqref{RPDDDLQRSDP}, and the purple one is the spectral norm of the feedback gain based on  \eqref{RPDDDLQRSDP} with a fixed $T$.}
  \label{fig_optimization_variable_change}
\vspace{-0.3cm}
\end{figure}

\begin{figure}
  \centering
  \includegraphics[scale=0.38]{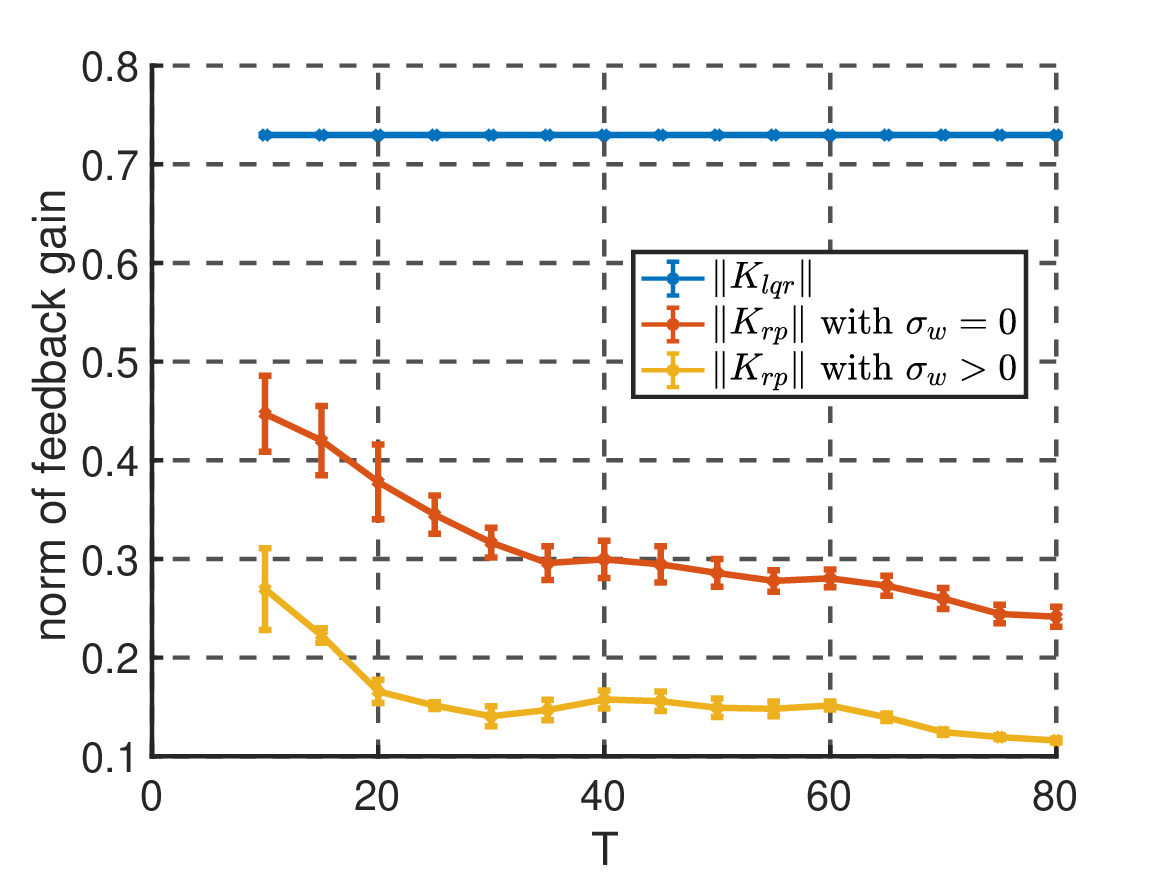}
  \caption{The only different setup of this figure with Fig. \ref{fig_gain_change_fixed_eta} is that we increase the regularization parameter $\eta$ in RP DDD LQR with the trajectory length $T$.}
  \label{fig_gain_change_in_eta}
\vspace{-0.3cm}
\end{figure}

\section{Conclusion and Future Work}
\label{sec_conclusion_future_work}
In this paper, we provide statistical analysis for two direct data-driven LQR methods in the presence of noise. Our results indicate that these methods are not statistically consistent. Therefore a ``certainty equivalence" approach that uses the original SDPs with noisy data is not appropriate. This is in contrast to model-based techniques, where certainty equivalence is known to be statistically consistent and sample-efficient. The identified limitations of the ``certainty equivalence" approach in direct data-driven control underscores the necessity of robust direct data-driven control methods \cite{van2020noisy}. Our future work will focus on understanding the statistical properties of such methods.

\section*{Acknowledgments}  
We thank the authors of \cite{de2021low} for their comments clarifying the relation of their results with ours. We also thank Prof. Vijay Subramanian of University of Michigan for helpful discussions on probabilistic convergence.


\bibliographystyle{IEEEtran}

 \appendix

In this lemma, we establish that continuous random variables, in particular Gaussian random variables, almost surely lead to persistently exciting inputs.
 \begin{lemma}
 \label{persistentinputnoise} 
Consider the input trajectory matrix $\V_0(T)$ defined in \eqref{eq:relationDataMatrices} with $\sigma_w>0$.  Let $T\geq (m+n)(n+1)+n$ and define $\Hb_T := \Hc_{n+1} \left(\V_0(T) \right)$.  Then, we have 
$$\mathbb{P}_T\left(\operatorname{rank}\left(\Hb_T\right)= (m+n)(n+1)\right)=1.$$
That is, $\V_0(T)$ is persistently exciting of order $n+1$ with probability one.
 \end{lemma}
 \begin{proof}   Note that, by definition of the Hankel matrix, $\Hb_T$ has $(m+n)(n+1)$ rows and $T-n$ columns. 
 We will show the rank condition holds for $T^*=(m+n)(n+1)+n$ (i.e., $\Hb_{T^*}$ is a square matrix), which implies the same for any $T\geq T^*$ by $\Hb_{T^*}$ being the first $T^*$ columns of $\Hb_{T}$. Let $g  := \det \left(\Hb_{T^*} \right).$ Since determinant $g$ is a non-constant polynomial function of the elements $\bu_0,\bu_1,\dots,\bu_{T^*-1},\bw_0,\bw_1,\dots,\bw_{T^*-1}$ of $\V_0(T^*)$ and all these basic random variables are continuous Gaussian random variables, the probability density function of $g $ is also continuous. Then, we have that $\Pb_{T^*} (g =0) = \Pb_{T^*} (g \leq 0) - \Pb_{T^*} (g <0)=0.$ This means that $\Pb_{T^*} (g \neq 0)=1$. Therefore, $\Hb_{T^*}$ is full rank with probability one, that is, $\Pb_{T^*}\left( \text{rank} \left(\Hb_{T^*}   \right)=(m+n)(n+1)\right) =1.$ 
 \end{proof}

\end{document}